\documentclass{amsart}

\usepackage{hyperref}

\theoremstyle{definition}

\theoremstyle{remark}

 \newtheorem{thm}{Theorem}[section]
 \newtheorem{cor}[thm]{Corollary}
 \newtheorem{lem}[thm]{Lemma}
 \newtheorem{prop}[thm]{Proposition}

\theoremstyle{remark}
\newtheorem{rem}[thm]{Remark}

\theoremstyle{definition}
\newtheorem{defn}[thm]{Definition}

\numberwithin{equation}{section}

  \allowdisplaybreaks

\begin{document}

\title{Blowup of Solutions of the Hydrostatic Euler Equations}


\author{Tak Kwong WONG}
\address{Department of Mathematics\\ University of Pennsylvania\\
David Rittenhouse Lab.\\
209 South 33rd Street \\
 Philadelphia, PA 19104-6395, USA}

\email{takwong@math.upenn.edu}

\keywords{formation of singularity, ill-posedness, hydrostatic approximation}
\subjclass[2010]{Primary 35Q31; Secondary 35A01, 35L04, 35L60, 35Q35, 76B99}

\date{\today}



\begin{abstract}
In this paper we prove that for a certain class of
initial data, smooth solutions of the hydrostatic Euler equations
blow up in finite time.
\end{abstract}

\maketitle

\tableofcontents

 \newcommand{\RE}{\mathbb{R}}
 \newcommand{\abs}[1]{\left\vert#1\right\vert}


\section{Introduction}

  We consider the hydrostatic Euler equations\footnotemark \footnotetext{In the literature, it is also called homogeneous hydrostatic equations (see ~\cite{Bre99} for example) or inviscid Prandtl's equations (see ~\cite{E00} for example).} in a two-dimensional tube $\RE\times[0,1]=\{(x,y);0\le y \le 1\}:$

  \begin{align}\label{e:HHE}
  u_t+uu_x+vu_y=-p_x,
  \end{align}
  \begin{align}\label{e:IncomEqt}
  u_x+v_y=0,
  \end{align}
  with given initial data
  \begin{align}\label{e:InitialData}
  u(0,x,y)=u_0(x,y),
  \end{align}
  and boundary condition
  \begin{align}\label{e:bdryCon}
  v(t,x,0)=v(t,x,1)=0,
  \end{align}
  where $(u,v,p)=(u(t,x,y),v(t,x,y),p(t,x))$ are unknowns.

  This system, which describes the leading behavior of an ideal flow moving in a very narrow domain $\RE\times [0,\epsilon]$, can be derived formally by the least action principle \cite{Bre08} or a rescaled limit \cite{Lio96}(\S4.6). The rescaled limit, called the hydrostatic limit in the literature, can be rigorously justified for the periodic flows under the local Rayleigh condition ~\cite{Bre03,Gre99,MW}.

  At the present time, there are no global existence results for \eqref{e:HHE}-\eqref{e:bdryCon}. Even for the local well-posedness, there are only three results. The first result was obtained under the local Rayleigh condition by Y. Brenier in \cite{Bre99}, in which the local existence of a special class of classical solutions was proved by applying a semi-Lagrangian reformulation. Another approach is using the energy method, the authors in \cite{MW} recently established the local well-posedness of $x$-periodic $H^s$ solutions under the local Rayleigh condition. Without this condition, the local existence and uniqueness of analytic solutions were also proved by a Cauchy-Kowalevski type argument in \cite{KTVZ11}.

  Regarding the blowup, there is a recent result \cite{CINT} which proved that for a certain class of initial data, spatially symmetric smooth solutions will blow up in finite time. 

  Our blowup result, which does not rely on the symmetry, is an analogous result for the unsteady Prandtl equations by W. E and B. Engquist ~\cite{EE97}. We prove that a smooth solution with certain class of initial data will blow up in finite time, see theorem \ref{t:main} below for the details. The blowup is either in
  $\nobreak{\max\{\| u(t)\|_{L^\infty},\| p_x (t)\|_{L^\infty}\}}$ or $\| u_x(t)\|_{L^\infty}$. The first case corresponds to the infinite horizontal velocity or pressure gradient, which is non-physical in certain sense. The second case corresponds to the formation of singularity.

  The main novelty of this proof is to ``freeze" a smooth solution by a classical invariant transformation, which will be given in section \ref{s:proof}. After ``freezing" the solution, we will apply an a priori estimate on the second derivative of the pressure term, which is a simple consequence of lemma \ref{l:1third} below, to derive a Ricatti type inequality. This provides the blowup. Finally, for the sake of self-containedness, the proof of a technical lemma will be given in section \ref{s:pflemma}.


\section{Main Theorem and Its Proof}
  The prime objective of this paper is to prove the following

  \begin{thm}[Blowup]\label{t:main}
    Let $(u,v,p)$ be a smooth solution to \eqref{e:HHE}-\eqref{e:bdryCon}. Suppose there exist a position $\hat{x}\in\RE$ and a constant horizontal velocity $\hat{u}\in\RE$ such that
    the initial data $u_0$ satisfies the following properties at $x = \hat{x}$:
    \begin{align}
      u_0(\hat{x},y)\equiv\hat{u}, & \text{ for all } y\in[0,1], \label{e:u0}\\
      u_{0xy}(\hat{x},0)=0, & \text{ and } \label{e:u0xy}\\
      u_{0xyy}(\hat{x},y)<0, & \text{ for all } y\in(0,1) \label{e:u0xyy}.
    \end{align}
    Then there exists a finite time $T>0$ such that either
    \begin{align}\label{e:limu&px}
     \lim_{t\to T^-} \max\{\| u(t) \|_{L^\infty},\| p_x(t)\|_{L^\infty}\}=+\infty,
    \end{align}

    or
    \begin{align}\label{e:limux}
      \lim_{t\to T^-} u_x(t,X(t,\hat{x},1),1)=-\infty,
    \end{align}
    where $X(t,\hat{x},1)$ is the $x$-component of the characteristic\footnote{For the precise definition of $X$,
    see definition \ref{d:XY eqn} below.}
    starting from $(\hat{x},1)$.
  \end{thm}

  \begin{proof}
      Our strategy is to show that $u_x$ blows up in finite time assuming that both $u$ and $p_x$ stay finite. That is,
      we will prove \eqref{e:limux} provided that \[\max_{t\geq0} \{\| u(t)\|_{L^\infty},\| p_x(t)\|_{L^\infty}\} < +\infty.\]\\\\
    \underline{Step1:} (Simplify the Problem)\\

    The first step is to reduce our problem by using the following ``freezing" lemma.
  \begin{lem}[Freezing the Solution]\label{l:2.2}
      Let $(u,v,p)$ be a smooth solution to \eqref{e:HHE}-\eqref{e:bdryCon} with the initial property \eqref{e:u0}. Assume that there exists a
      constant $M$ such that
      \[\max_{t\geq0} \{\| u(t)\|_{L^\infty},\| p_x(t)\|_{L^\infty}\}\le M < +\infty.\]
      Then there exists a classical invariant transformation
      \[(u,v,p,t,x,y)\mapsto (\tilde{u},\tilde{v},\tilde{p},\tilde{t},\tilde{x},\tilde{y})\]
      such that
      \begin{itemize}
        \item[(i)] $(\tilde{u},\tilde{v},\tilde{p})$ satisfies \eqref{e:HHE}, \eqref{e:IncomEqt} and \eqref{e:bdryCon},
        \item[(ii)] $\tilde{u}_{\tilde{x}} (\tilde{t},\tilde{x},\tilde{y})=u_x(t,x,y),$
        \item[(iii)] $\tilde{u}(\tilde{t},0,\tilde{y})\equiv 0.$
      \end{itemize}
      In addition, if the initial data $u_0$ also satisfies properties \eqref{e:u0xy}-\eqref{e:u0xyy}, then we have
       \begin{itemize}
        \item[(iv)] $\tilde{u}_{\tilde{x}\tilde{y}}(0,0,0)=0,$
        \item[(v)] $\tilde{u}_{\tilde{x}\tilde {y}\tilde{y}}(0,0,\tilde{y})<0$ for all $\tilde{y}\in(0,1)$.
      \end{itemize}
    \end{lem}

    The proof of lemma \ref{l:2.2} will be given in section \ref{s:proof}. The importance of lemma \ref{l:2.2} is that without loss of generality, we may assume that $\hat{x}=0$ and
    \begin{align}\label{e:u}
      u(t,0,y)\equiv 0, \quad \text{ for all }\;(t,y)\in\RE^+\times[0,1].
    \end{align}

    A direct consequence of \eqref{e:u} is
    \begin{align}\label{e:uy}
      u_y(t,0,y)\equiv 0.
    \end{align}

    Moreover, using \eqref{e:u} and lemma \ref{l:2.2}, we can rewrite our aim \eqref{e:limux} as
    \begin{align}\label{e:limuxneg}
      \lim_{t\to T^-} u_x(t,0,1)=-\infty.
    \end{align}
    \\
  \underline{Step2:}(Simplify the System)\\

  In this step, we will further simplify the system as follows.

  Differentiating \eqref{e:HHE} with respect to $x$, we obtain
  \begin{align}\label{e:DHHE}
    u_{xt}+uu_{xx}+u_x^2+vu_{xy}+v_xu_y=-p_{xx}.
  \end{align}
  Integrating \eqref{e:DHHE} with respect to $y$ over $[0,1]$, using \eqref{e:IncomEqt}, \eqref{e:bdryCon} and the fact that $p$ is independent of $y$, we obtain an integral representation
  \begin{align}\label{e:pxx}
    -p_{xx}=2\int^1_0 uu_{xx}+u_x^2\;dy.
  \end{align}
  Let $a(t,y):=-u_x(t,0,y), \;a_0(y):=-u_{0x}(0,y) $ and v$(t,y):=v(t,0,y).$ Then restricting \eqref{e:DHHE} to $x=0$, and using \eqref{e:u}, \eqref{e:uy} and \eqref{e:pxx}, we have
  \begin{align}\label{e:at}
    a_t+\text{v}a_y=a^2-2\int^1_0 a^2\;dy.
  \end{align}
  It follows from the definition of $a$ and the incompressibility condition \eqref{e:IncomEqt} that
  \begin{align}
    \text{v}_y=a,
  \end{align}
  and hence, the boundary condition \eqref{e:bdryCon} implies that
  \begin{align}\label{e:inta}
    \text{v}(t,0) = \text{v}(t,1) = 0 \text{ and } \int^1_0 a\;dy=0.
  \end{align}
  Furthermore, the initial data \eqref{e:InitialData} gives
  \begin{align}\label{e:a0y}
    a(0,y)=a_0(y),
  \end{align}
  and hypotheses \eqref{e:u0xy} and  \eqref{e:u0xyy} become
  \begin{align}\label{e:a0y0}
   a_{0y}(0)=0,
  \end{align}
  \begin{align}\label{e:a0yy}
    a_{0yy}>0.
  \end{align}
  Lastly, our aim \eqref{e:limuxneg} becomes
  \begin{align}\label{e:limat1}
    \lim_{t\to T^-} a(t,1)=+\infty.
  \end{align}
  \\
  \underline{Step3:}(Blowup Estimate)\\

  In this step we need two lemmas as follows.

  \begin{lem}\label{l:ayt0}
    Let $a:\RE^+\times[0,1]\to \RE $ be a smooth solution to
    $(\ref{e:at})-(\ref{e:a0y})$. If $a_0$ satisfies
    $(\ref{e:a0y0})-(\ref{e:a0yy})$, then
    \[a_y(t,0) \equiv 0 \text{ and } a_{yy}>0.\]
  \end{lem}
  \begin{lem}\label{l:1third}
    Let $f:[0,1]\to\RE $ be a $C^2$ function with the following
    properties:
    \begin{itemize}
      \item[(i)] $f'(0)=0 \quad\text{and}\quad f''>0,$\\
      \item[(ii)] $ \int^1_0 f\;dy=0.$
    \end{itemize}
    Then $f(1)>0$ and
    \begin{align}\label{e:1third}
      \int^1_0 f^2\;dy \le \frac{1}{3} f(1)^2.
    \end{align}
  \end{lem}

    The proof of lemma \ref{l:ayt0} is based on the characteristics method and will be given in section \ref{s:pflemma}. On the other hand, lemma \ref{l:1third} is just an elementary property for the convex functions, so we leave it for the reader. For a proof of lemma \ref{l:1third}, we refer the reader to lemma 3.4.3 of \cite{Won10} for instance.

     Assuming these lemmas for the moment, we can prove the blowup \eqref{e:limat1} as follows.

     It follows from lemma \ref{l:ayt0} and \eqref{e:inta} that $a(t,\cdot)$ satisfies the hypotheses of lemma \ref{l:1third}, so the $L^2$ estimate \eqref{e:1third} implies that
    \[a_t+\text{v}a_y=a^2-2\int^1_0 a^2\;dy\ge a^2-\frac{2}{3}a(t,1)^2.\]

    Since v$(t,1)=0,$ we obtain a Ricatti type inequality
    \begin{align*}
      a_t(t,1)\ge \frac{1}{3} a(t,1)^2,
    \end{align*}
    and hence,
    \begin{equation}\label{e:at1 est}
        a(t,1)\ge \frac{3\;a_0(1)}{3-a_0(1)\;t}.
    \end{equation}
    Finally, applying lemma \ref{l:1third} to $a_0$, we have  $a_0(1)>0$, which and \eqref{e:at1 est} imply that there
    exists a finite time $T>0$ such that \eqref{e:limat1} holds. This completes the proof.

  \end{proof}


\section{Basic Properties and Classical Invariant Transformations
}\label{s:proof}
  The main purpose of this section is to prove lemma \ref{l:2.2}. To do this, we will first study the basic properties of smooth solutions to \eqref{e:HHE}-\eqref{e:bdryCon} as follows.

  In general, smooth solutions to \eqref{e:HHE}-\eqref{e:bdryCon} are not unique because the classical invariant transformation group
  \begin{align}\label{e:Tgroup}
    \begin{cases}
      \tilde{u}:=u-g' & \quad\quad\tilde{t}:=t\\
      \tilde{v}:=v & \quad\quad\tilde{x}:=x-g \quad\quad\quad g:=g(t)\\
      \tilde{p}:=p+xg'' & \quad\quad\tilde{y}:=y
    \end{cases}
  \end{align}
  will produce a new solution of \eqref{e:HHE}-\eqref{e:bdryCon} if $g(0)=g'(0)=0$. Here, we can also use this transformation group to
  prove lemma \ref{l:2.2} if we choose $g$ appropriately.

    In order to choose a suitable $g$, we have to study the characteristics given by the following

  \begin{defn} \label{d:XY eqn}
  The functions $X(t,x_0,y_0)$ and $Y(t,x_0,y_0)$ are called the $x$-component and $y$-component of the characteristic starting from $(x_0,y_0)$ respectively if they satisfy
  \begin{align*}
    \begin{cases}
      \dot{X}(t,x_0,y_0)=u(t,X(t,x_0,y_0),Y(t,x_0,y_0))\\
      X(0,x_0,y_0)=x_0\\
      \dot{Y}(t,x_0,y_0)=v(t,X(t,x_0,y_0),Y(t,x_0,y_0))\\
      Y(0,x_0,y_0)=y_0,
    \end{cases}
  \end{align*}
  where the dot represents $\frac{d}{dt}$. We may also write $X(t)$ and $Y(t)$ if there is no ambiguity.
  \end{defn}

  The characteristics $X(t)$ of a smooth solution $(u,v,p)$ of the hydrostatic Euler equations \eqref{e:HHE}-\eqref{e:IncomEqt} has an interesting property, which can be stated as follows:

  \begin{prop}\label{p:prop1}
    If $u_0(x_0,y_0)=u_0(x_0,y_1)$, then
    \begin{align}\label{e:xy0eqxy1}
      X(t,x_0,y_0)=X(t,x_0,y_1), \text{ and}
    \end{align}
    \begin{align}\label{e:uy0equy1}
      u(t,X(t,x_0,y_0),Y(t,x_0,y_0))=u(t,X(t,x_0,y_1),Y(t,x_0,y_1)).
    \end{align}
  \end{prop}
  \begin{proof}
    The equality \eqref{e:xy0eqxy1} follows directly from the uniqueness of ODE because both $X(t,x_0,y_0)$ and $X(t,x_0,y_1)$ satisfy the same ODE
      \[\ddot{X}=-p_x(t,X)\]
    with the same initial data
      \[\begin{cases}
        X(0)=x_0\\
        \dot{X}(0)=u_0(x_0,y_0).
      \end{cases}\]
    Equality \eqref{e:uy0equy1} can be obtained by differentiating \eqref{e:xy0eqxy1} with respect to $t$.
  \end{proof}

  Proposition \eqref{p:prop1} and the boundary condition \eqref{e:bdryCon} give us a useful
  \begin{cor}\label{c:indept}
    If $u_0(\hat{x},y_0)$ is independent of $y_0$, then $X(t,\hat{x},y_0)$ and $u(t,X(t,\hat{x},y_0),y)$ are also independent of $y_0$ and $y$.
  \end{cor}

  Now, using the knowledge above, we are going to prove lemma \ref{l:2.2}.

  From the hypothesis \eqref{e:u0}, we know that $u_0(\hat{x},y_0)$ is independent of $y_0$, so by corollary \ref{c:indept}, $X(t,\hat{x},y_0)$ is independent of $y_0$. Therefore, we can take $g(t):=X(t,\hat{x},y_0)$ for any $y_0\in [0,1]$. Since $\abs{g'(t)}=\dot{\abs{X}}=\abs{u(t,X,Y)}\le M$ and  $\abs{g''(t)}=\ddot{\abs{X}}=\abs{p_x(t,X)}\le M,$ we can apply the invariant transformation \eqref{e:Tgroup} as long as $(u,v,p)$ is smooth.

  To complete the proof, we have to check that this transformation satisfies $(i)-(v)$ in lemma \ref{l:2.2}. By direct computations, $(i) - (ii)$ and $(iv) - (v)$ can be checked easily by using the following two facts:\\
    \begin{itemize}
      \item[(a)] $\partial_{\tilde{t}} = \partial_t+g'\partial_x,\;
      \partial_{\tilde{x}}=\partial_x \text{ and }
      \partial_{\tilde{y}}=\partial_y,$\\
      \item[(b)] $g(0)=\hat{x}.$\\
    \end{itemize}

  Lastly, $(iii)$ holds because
  \begin{align*}
    \tilde{u}(\tilde{t},0,\tilde{y}) = &\;
    u(\tilde{t},g(\tilde{t}),\tilde{y})-g'(\tilde{t})\\
    = &\;
    u(\tilde{t},g(\tilde{t}),\tilde{y})-u(\tilde{t},g(\tilde{t}),
    Y(\tilde{t},\hat{x},y_0))\\
    = &\; 0,
  \end{align*}
  where we applied corollary \ref{c:indept} in the last equality.

  \begin{rem}
    One may conclude this section by a thought experiment: Imagine that there is a water flow $(u,v,p)$ in an infinitely long river   $\RE\times[0,1]$. As a stationary observer, you can stand at a fixed point on the river bank, say the origin, then you will see the flow as
    $(u,v,p)$. On the other hand, you can move parallel to the river bank with velocity $g'(t)$, then you will see the flow as    $(\tilde{u},\tilde{v},\tilde{p})$. In other words, both $(u,v,p)$ and $(\tilde{u},\tilde{v},\tilde{p})$ describe the same physical phenomenon but with different reference frames, so lemma \ref{l:2.2} is just stating that if the initial horizontal velocity $u_0$ has a line $x=\hat{x}$ with a constant
    horizontal velocity $\hat{u}$, then you can freeze this constant horizontal velocity line provided that you move appropriately.

    Another physical interpretation can be made from this thought experiment is that the blowup \eqref{e:limu&px} of $u$ and $p_x$ can be seen as the consequence of an observer moving at an infinite speed or accelerating at an infinite rate respectively.
  \end{rem}

\section{Proof of Technical Lemma \ref{l:ayt0}}\label{s:pflemma}

In this section, we will prove the technical lemma \ref{l:ayt0}.

\begin{proof}[Proof of lemma \ref{l:ayt0}:]

  First, differentiating \eqref{e:at} with respect to $y$ once, we have
  \[a_{yt}+\text{v}a_{yy}=aa_y,\]
  which implies $a_y(t,0)\equiv 0$ provided that $a_{0y}(0)=0$ and v$(t,0)\equiv 0.$ Second, differentiating \eqref{e:at} with respect to $y$ twice, we
  obtain
  \[a_{yyt}+\text{v}a_{yyy}=a^2_y\ge 0,\]
  which means that $a_{yy}$ is increasing along every characteristic, so \eqref{e:a0yy} implies $a_{yy}>0.$

\end{proof}

  \bibliographystyle{amsalpha}
  \bibliography{HEEReference}

\end{document}